\documentclass[graybox]{svmult}

\usepackage{mathptmx}       % selects Times Roman as basic font
\usepackage{helvet}         % selects Helvetica as sans-serif font
\usepackage{courier}        % selects Courier as typewriter font
\usepackage{type1cm}        % activate if the above 3 fonts are
\usepackage{makeidx}         % allows index generation
\usepackage{graphicx}        % standard LaTeX graphics tool
\usepackage{multicol}        % used for the two-column index
\usepackage[bottom]{footmisc}% places footnotes at page bottom

\makeindex             % used for the subject index
\def\deltas{{{\delta_*}}}
\def\cA{{\mathcal A}}
\def\ucA{\underline{\mathcal A}}
\def\barS{{\bar{S}}}
\def\toi{{\,t \in [0,\infty)}}
\def\BE{\begin{equation}}
\def\EE{\end{equation}}
\def\l({\left(}
\def\r){\right)}
\def\toT{\, t \in [0,T]}
\def\sbg|{\,\bigg|\,}
\def\bX{{{\mbox {\boldmath $X$}}}}
\def\bM{{{\mbox {\boldmath $M$}}}}
\def\bx{{{\mbox {\boldmath $x$}}}}
\def\bb{{{\mbox {\boldmath $b$}}}}
\def\ba{{{\mbox {\boldmath $a$}}}}
\def\bW{{{\mbox {\boldmath $W$}}}}
\def\psdmatd{\overline{S^+_d}}
\def\bix{{{\mbox {\boldmath \scriptsize{$x$}}}}}
\def\matd{\mathcal{M}_d}
\def\bB{{{\mbox {\boldmath $B$}}}}
\def\bA{{{\mbox {\boldmath $A$}}}}
\def\bI{{{\mbox {\boldmath $I$}}}}
\def\vo{{{\mbox {\boldmath $0$}}}}
\def\vSigma{{{\mbox {\boldmath $\Sigma$}}}}
\def\pdmatd{S^+_d}
\def\WISt{WIS_d(\bx,\alpha,\bb,\ba;t)}
\def\vPsi{{{\mbox {\boldmath$\Psi$}}}}  % if in formula
\def\cN{{\mathcal N}}
\def\vTheta{{{\mbox {\boldmath$\Theta$}}}}  % if in formula
\def\bS{{{\mbox {\boldmath $S$}}}}
\def\bY{{{\mbox {\boldmath $Y$}}}}
\def\bV{{{\mbox {\boldmath $V$}}}}
\def\bQ{{{\mbox {\boldmath $Q$}}}}

\begin{document}

\title*{Computing Functionals of Multidimensional Diffusions via Monte Carlo Methods}
\author{Jan Baldeaux and Eckhard Platen}
\institute{Jan Baldeaux \at University of Technology Sydney,
Finance Discipline Group, PO Box 123, Broadway, NSW,
2007, Australia \email{Jan.Baldeaux@uts.edu.au}
\and Eckhard Platen \at University of Technology Sydney,
Finance Discipline Group and School of Mathematical Sciences, PO Box 123, Broadway, NSW,
2007, Australia \email{Eckhard.Platen@uts.edu.au}}
\maketitle

\abstract*{Each chapter should be preceded by an abstract (10--15 lines long) that summarizes the content. The abstract will appear \textit{online} at \url{www.SpringerLink.com} and be available with unrestricted access. This allows unregistered users to read the abstract as a teaser for the complete chapter. As a general rule the abstracts will not appear in the printed version of your book unless it is the style of your particular book or that of the series to which your book belongs.
Please use the 'starred' version of the new Springer \texttt{abstract} command for typesetting the text of the online abstracts (cf. source file of this chapter template \texttt{abstract}) and include them with the source files of your manuscript. Use the plain \texttt{abstract} command if the abstract is also to appear in the printed version of the book.}

\abstract{We discuss suitable classes of diffusion processes, for which functionals relevant to finance can be computed via Monte Carlo methods. In particular, we construct exact simulation schemes for processes from this class. However, should the finance problem under consideration require e.g. continuous monitoring of the processes, the simulation algorithm can easily be embedded in a multilevel Monte Carlo scheme. We choose to introduce the finance problems under the benchmark approach, and find that this approach allows us to exploit conveniently the analytical tractability of these diffusion processes.}

\section{Introduction}
\label{secintro}

In mathematical finance, the pricing of financial derivatives can under suitable conditions be shown to amount to the computation of an expected value, see e.g. \cite{MR05}, \cite{PH10}. Depending on the financial derivative and the model under consideration, it might not be possible to compute the expected value explicitly, however, numerical methods have to be invoked. A candidate for the computation of such expectations is the Monte Carlo method, see e.g. \cite{Boyle77}, \cite{G04}, and \cite{KP99}. Applying the Monte Carlo method typically entails the sampling of the distribution of the relevant financial state variables, e.g. an equity index, a short rate, or a commodity price. It is then, of course, desirable to have at one's disposal a recipe for drawing samples from the relevant distributions. In case these distributions are known, one refers to exact simulation schemes, see e.g. \cite{PB10}, but also  \cite{BeskosRo05}, \cite{BeskosPaRo06}, \cite{BeskosPaRo08}, and \cite{Chen08}, for further references on exact simulation schemes. If exact simulation schemes are not applicable, discrete time approximations, as analyzed in \cite{KP99} and \cite{PB10} become relevant. In recent years, it has been shown under certain assumptions that using the multilevel Monte Carlo method, see \cite{Giles08} and also \cite{H98}, \cite{H01}, the standard Monte Carlo convergence rate, achieved by exact simulation schemes, can be recovered.

For modeling financial quantities of interest, it is important to know a priori if exact simulation schemes exist, so that financial derivatives can be priced, even if expected values cannot be computed explicitly. In this paper, we discuss classes of stochastic processes for which this is the case. For one-dimensional diffusions, Lie symmetry analysis, see \cite{BlumanKum89}, and \cite{O93} turns out to be a useful tool. Besides allowing one to discover transition densities, see \cite{CraddockPla04}, it also allows us to compute Laplace transforms of important multidimensional functionals, see e.g. \cite{CraddockLe09}. In particular, we find that squared Bessel processes fall into the class of diffusions that can be handled well via Lie symmetry methods.

The Wishart process, \cite{Bru91}, is the multidimensional extension of the squared Bessel process. It turns out, see \cite{GouSuf03} and \cite{GouSuf04}, that Wishart processes are affine processes, i.e. their characteristic function is exponentially affine in the state variables. We point out that in \cite{GouSuf03}, and \cite{GouSuf04} the concept of an affine process was generalized from real-valued processes to matrix-valued processes, where the latter category covers Wishart processes. Furthermore, the characteristic function can be computed explicitly, see \cite{GouSuf03}, and \cite{GouSuf04}. Finally, we remark that in \cite{AhdidaAlfonsi10} an exact simulation scheme for Wishart processes was presented.

Modeling financial quantities, one aims for models which provide an accurate reflection of reality, whilst at the same time retaining analytical tractability. The benchmark approach, see \cite{PH10}, offers a unified framework to derivative pricing, risk management, and portfolio optimization. It allows us to use a much wider range of empirically supported models than under the classical no-arbitrage approach. At the heart of the benchmark approach sits the growth optimal portfolio (GOP). It is the portfolio which maximizes expected log-utility from terminal wealth. In particular, the benchmark approach uses the GOP as num\'eraire and the real world probability for taking expectations. We find that the class of processes for which exact simulation is possible is easily accommodated under the benchmark approach, which we illustrate using examples.

The remaining structure of the paper is as follows: In Section \ref{secBMApproach} we introduce the benchmark approach using a particular model for illustration, the minimal market model (MMM), see \cite{PH10}. Section \ref{secLieSym} introduces Lie symmetry methods and discusses how they can be used in the context of the benchmark approach. Section \ref{secWishproc} presents Wishart processes and shows how they can be used to extend the MMM. Section \ref{secconc} concludes the paper.

\section{Benchmark Approach} \label{secBMApproach}

The GOP plays a pivotal role as benchmark and num\'eraire under the benchmark approach. It also enjoys a prominent position in the finance literature, see \cite{Kelly56}, but also \cite{Breiman60}, \cite{Latane59}, \cite{Markowitz76}, \cite{Long90}, \cite{MacLeanThoZie11}, and \cite{Thorp61}. The benchmark approach uses the GOP as the num\'eraire. Since the GOP is the num\'eraire portfolio, see \cite{Long90}, contingent claims are priced under the real world probability measure. This avoids the restrictive assumption on the existence of an equivalent risk-neutral probability measure. We remark, it is argued in \cite{PH10} that the existence of such a measure may not be a realistic assumption. Finally, we emphasize that the benchmark approach can be seen as a generalization of risk-neutral pricing, as well as other pricing approaches, such as actuarial pricing, see \cite{PH10}.

To fix ideas in a simple manner, we model a well-diversified index, which we interpret as the GOP, using the stylized version of the MMM, see \cite{PH10}. Though parsimonious, this model is able to capture important empirical characteristics of well-diversified indices. It has subsequently been extended in several ways, see e.g. \cite{PH10}, and also \cite{BIP12}. To be precise, consider a filtered probability space $(\Omega,\cA,\ucA,P)$, where the filtration $\ucA=(\cA_t)_\toi$ is
assumed to satisfy the usual conditions, which carries for simplicity one source of uncertainty, a standard Brownian motion $W=\{W_t,\toi\}$. The deterministic savings account is modeled using the differential equation
\begin{displaymath}
dS^0_t = r\,S^0_t\,dt \, ,
\end{displaymath}
for $\toi$ with $S^0_0=1$, where $r$ denotes the constant short rate. Next, we introduce the model for the well diversified index, the GOP $S^{\deltas}_t$, which is given by the expression
 \BE
 \label{eqdefGOP}
 S^\deltas_t = S^0_t\,\barS^\deltas_t = S^0_t\,Y_t\,\alpha^\deltas_t \, .
 \EE
Here $Y_t = \frac{\alpha^{\deltas}_t}{\barS^\deltas_t}$ is a square-root process of dimension four, satisfying the stochastic differential equation (SDE)
\BE \label{eqsqrootproc}
dY_t = (1-\eta\,Y_t)\,dt + \sqrt{Y_t}\,dW_t \, ,
\EE
for $\toi$ with initial value $Y_0 > 0$ and net growth rate $\eta > 0$. The deterministic function of time $\alpha^\deltas_t $ is given by the exponential function
\begin{displaymath}
\alpha^{\deltas}_t = \alpha_0 \exp \left\{ \eta t \right\} \, ,
\end{displaymath}
with scaling parameter $\alpha_0 > 0$. Furthermore, it can be shown by the It\^o formula that $\alpha^{\deltas}_t$ is the drift at time $t$ of the discounted GOP
\begin{displaymath}
\barS^\deltas_t := \frac{S^{\deltas}_t}{S^0_t} \, ,
\end{displaymath}
so that the parameters of the model are $S^{\deltas}_0$, $\alpha_0$, $\eta$, and $r$. We note that one obtains for the GOP the SDE
\BE \label{eqsdeGOP}
d S^{\deltas}_t = S^{\deltas}_t \l(  \l( r + \frac{1}{Y_t}  \r) dt + \sqrt{\frac{1}{Y_t}} d W_t   \r) \, ,
\EE
which illustrates the well-observed leverage effect, since as the index $S^{\deltas}_t$ decreases, its volatility $\frac{1}{\sqrt{Y_t}}=\sqrt{\frac{\alpha^{\deltas}_t}{\barS^\deltas_t}}$ increases and vice versa.

It is useful to define the transformed time $\varphi(t)$ as
\begin{displaymath}
\varphi(t) = \varphi(0) + \frac{1}{4} \int^t_0 \alpha^{\deltas}_s ds \, .
\end{displaymath}
Setting
\begin{displaymath}
X_{\varphi(t)} = \barS^\deltas_t \, ,
\end{displaymath}
we obtain the SDE
\BE \label{eqdiscountedGOPsqBessel}
d X_{\varphi(t)} = 4 d \varphi(t) + 2 \sqrt{X_{\varphi(t)}} d W_{\varphi(t)} \, ,
\EE
where
\begin{displaymath}
d W_{\varphi(t)} = \sqrt{\frac{\alpha^{\deltas}_t}{4} } d W_t
\end{displaymath}
for $t \in [0,\infty)$. This shows that $\,X=\{X_\varphi,\,\varphi \in [\varphi(0),\infty)\}$ is a time transformed squared Bessel process of dimension four and $\,W=\{W_\varphi,\,\varphi \in [\varphi(0),\infty)\}$ is a Wiener process in the transformed $\varphi$-time $\varphi(t) \in [\varphi(0), \infty)$, see \cite{RevuzYor99}. The merit of the dynamics given by (\ref{eqdiscountedGOPsqBessel}) is that transition densities of squared Bessel processes are well studied; in fact we derive them in Section \ref{secLieSym} using Lie symmetry methods.

We remark that the MMM does not admit a risk-neutral probability measure because the Radon-Nikodym derivative $\Lambda_t = \frac{\barS^\deltas_0}{\barS^\deltas_t}$ of the putative risk-neutral measure, which is the inverse of a time transformed squared Bessel process of dimension four, is a strict local martingale and not a martingale, see \cite{RevuzYor99}. On the other hand, $S^\deltas$, is the num\'eraire portfolio, and thus, when used as num\'eraire to denominate any nonnegative portfolio, yields a supermartingale under the real-world probability measure $P$. This implies that the financial market under consideration is free of those arbitrage opportunities that are economically meaningful in the sense that they would allow to create strictly positive wealth out of zero initial wealth via a nonnegative portfolio, that is, under limited liability, see \cite{LoewensteinWi00a} and \cite{PH10}. This also means that we can price contingent claims under $P$ employing $S^{\deltas}$ as the num\'eraire. This pricing concept is referred to as real-world pricing, which we now recall, see \cite{PH10}: for a nonnegative contingent claim with payoff $H$ at maturity $T$, where $H$ is $\cA_T$-measurable, and $E \l( \frac{ H }{S^\deltas_T} \r) < \infty$, we define the value process at time $\toT$ by
\BE \label{eqrealworldprice} V_t := S^{\deltas}_t E \l( \frac{H}{S^{\deltas}_T} \sbg|\cA_t  \r) \, .
\EE
Note that since $V_T=H$, the benchmarked price process $\frac{V_t}{S^{\deltas}_t}$ is an $\l(\ucA, P\r)$-martingale. Formula (\ref{eqrealworldprice}) represents the real-world pricing formula, which provides the minimal possible price and will be used in this paper to price derivatives. If the expectation in equation (\ref{eqrealworldprice}) cannot be computed explicitly, one can resort to Monte Carlo methods. In that case, it is particularly convenient, if the relevant financial quantities, such as $S^{\deltas}_T$ can be simulated exactly. In the next section, we derive the transition density of $S^{\deltas}$ via Lie symmetry methods, which then allows us to simulate $S^{\deltas}_T$ exactly. Note, in Section \ref{secWishproc}, we generalize the MMM to a multidimensional setting and present a suitable exact simulation algorithm.

\section{Lie Symmetry Methods} \label{secLieSym}

The aim of this section is to present Lie symmetry methods as an effective tool for designing tractable models in mathematical finance. Tractable models are, in particular, useful for the evaluation of derivatives and risk measures in mathematical finance. We point out that in the literature, Lie symmetry methods have been used to solve mathematical finance problems explicitly, see e.g. \cite{CraddockLe07}, and \cite{Itkin12}. Within the current paper we want to demonstrate that they can also be used to design efficient Monte Carlo algorithms for complex multidimensional functionals.

The advantage of the use of Lie symmetry methods is that it is straightforward to check whether the method is applicable or not. If the method is applicable, then the relevant solution or its Laplace transform has usually already been obtained in the literature or can be systematically derived. We will demonstrate this in finance applications using the benchmark approach for pricing.

We now follow \cite{CraddockLe09}, and recall that if the solution of the Cauchy problem
\begin{eqnarray}
u_t &=& b x^{\gamma} u_{x x} + f(x) u_x - g(x) u \, , \, x > 0 \, , \, t \geq 0 \, ,  \label{eqPDEFC1}
\\ u(x,0) &=& \varphi(x) \, , \, x \in \Omega =[0, \infty) \, ,
\end{eqnarray}
is unique, then by using the Feynman-Kac formula it is given by the expectation
\begin{displaymath}
u(x,t) = E \left( \exp \left( - \int^t_0 g(X_s) ds \right) \varphi(X_t) \right) \, ,
\end{displaymath}
where $X_0=x$, and the stochastic process $X= \left\{ X_t \, , \, t \geq 0 \right\}$ satisfies the SDE
\begin{displaymath}
d X_t = f(X_t) dt + \sqrt{2 b X^{\gamma}_t}  dW_t \, .
\end{displaymath}
We now briefly describe the intuition behind the application of Lie Symmetry methods to problems from mathematical finance, in particular, the integral transform method developed in \cite{Lennox12}, and the types of results this approach can produce. Lie's method allows us to find vector fields
\begin{displaymath}
\mathbf{v} = \xi(x,y,u) \partial_x + \tau(x,t,u) \partial_t + \phi(x,t,u) \partial_u \, ,
\end{displaymath}
which generate one parameter Lie groups that preserve solutions of (\ref{eqPDEFC1}). It is standard to denote the action of $\mathbf{v}$ on solutions $u(x,t)$ of (\ref{eqPDEFC1}) by
\begin{equation} \label{eqactiononu}
\rho ( \exp \epsilon \mathbf{v} ) u(x,t) = \sigma(x,t; \epsilon) u (a_1(x,t; \epsilon) , a_2 (x,t; \epsilon))
\end{equation}
for some functions $\sigma$, $a_1$, and $a_2$, where $\epsilon$ is the parameter of the group, $\sigma$ is referred to as the multiplier, and $a_1$ and $a_2$ are changes of variables of the symmetry. For the applications we have in mind, $\epsilon$ and $\sigma$ are of crucial importance, $\epsilon$ will play the role of the transform parameter of the Fourier or Laplace transform and $\sigma$ will usually be the Fourier or Laplace transform of the transition density. Following \cite{CraddockLe07}, we assume that (\ref{eqPDEFC1}) has a fundamental solution $p(t,x,y)$. For this paper, it suffices to recall that we can express a solution $u(x,t)$ of the PDE (\ref{eqPDEFC1}) subject to the initial condition $u(x,0)=f(x)$ in the form
\BE \label{eqtransformconnect2}
u(x,t) = \int_{\Omega} f(y) p(t,x,y) dy \, ,
\EE
where $p(t,x,y)$ is a fundamental solution of (\ref{eqPDEFC1}). %
%We refer the reader to Subsection \ref{subsecfunsolutionsPDE}, where the concept of a fundamental solution is recalled. Consequently, the function
%\BE \label{eqtransformconnect2}
%u(x,t) = \int_{\Omega} f(y) p(t,x,y) dy
%\EE
%solves the PDE \eqref{eqPDEFC1} with initial data $u(x,0)=f(x)$.
The key idea of the transform method is to connect (\ref{eqactiononu}) and (\ref{eqtransformconnect2}). Now consider a stationary, i.e. a time-independent solution, say $u_0(x)$. Of course, (\ref{eqactiononu}) yields
\begin{displaymath}
\rho \l( \exp \epsilon \mathbf{v} \r) u_0(x) = \sigma \l( x , t ; \epsilon \r) u_0 \l( a_1 (x,t ; \varepsilon ) \r) \, ,
\end{displaymath}
which also solves the initial value problem. We now set $t=0$ and use (\ref{eqactiononu}) and (\ref{eqtransformconnect2}) to obtain
\BE \label{eqtransformconnect3}
\int_{\Omega} \sigma(y,0,\epsilon) u_0 \l( a_1 \l( y , 0, ; \epsilon \r) \r) p \l( t,x,y \r) dy = \sigma \l( x,t ; \epsilon \r) u_0 \l( a_1 \l( x, t; \epsilon \r) \r) \, .
\EE
Since $\sigma$, $u_0$, and $a_1$ are known functions, we have a family of integral equations for $p(t,x,y)$. To illustrate this idea using an example, we consider the one-dimensional heat equation
\BE \label{eqonedimheat}
u_t = \frac{1}{2} g^2 u_{xx} \, .
\EE
We will show that if $u(x,t)$ solves (\ref{eqonedimheat}), then for $\epsilon$ sufficiently small, so does
\begin{displaymath}
\tilde{u} (t,z) = \exp \left\{ \frac{\epsilon t^2}{2 g^2} - \frac{z \epsilon}{g^2} \right\} u \l( z - t \epsilon , t \r) \, .
\end{displaymath}
Taking $u_0=1$, (\ref{eqtransformconnect3}) gives
\begin{displaymath}
\int^{\infty}_{-\infty} \exp \left\{ - \frac{y \epsilon}{g^2} \right\} p(t,x,y) dy = \exp \left\{ \frac{ t \epsilon^2 }{2 g^2} - \frac{x \epsilon}{g^2} \right\} \, .
\end{displaymath}
Setting $a=-\frac{\epsilon}{g^2}$, we get
\BE \label{eqmgfGaussian}
\int^{\infty}_{-\infty} \exp \{ a y \} p(t,x,y) dy = \exp \left\{ \frac{a^2 g^2 t}{2} + a x \right\} \, .
\EE
We recognize that (\ref{eqmgfGaussian}) is the moment generating function of the Gaussian distribution, so $p(t,x,y)$ is the Gaussian density with mean $x$ and variance $g^2 t$. We alert the reader to the fact that $\epsilon$ plays the role of the transform parameter and $\sigma$ corresponds to the moment generating function. Finally, we recall a remark from \cite{Craddock09}, namely the fact that Laplace and Fourier transforms can be readily obtained through Lie algebra computations, which suggests a deep relationship between Lie symmetry analysis and harmonic analysis. Lastly, we remark that in order to apply the approach, we require the PDE (\ref{eqPDEFC1}) to have nontrivial symmetries. The approach developed by Craddock and collaborators, see \cite{Craddock09}, \cite{CraddockDo10}, \cite{CraddockLe07}, \cite{CraddockLe09}, and \cite{CraddockPla04}, provides us with the following: A statement confirming if nontrivial symmetries exist and an expression stemming from (\ref{eqtransformconnect3}), which one only needs to invert to obtain $p(t,x,y)$. We first present theoretical results, and then apply these to the case of the MMM. Now we discuss the question whether the PDE (\ref{eqPDEFC1}) has nontrivial symmetries, see \cite{CraddockLe09}, Proposition 2.1.

\begin{theorem} \label{theoexistnontrivsym} If $\gamma \neq 2$, then the PDE
\BE \label{eqtheoexistnontrivsym}
u_t = b x^{\gamma} u_{x x} + f(x) u_x - g(x) u \, , \quad x \geq 0 \, , b > 0
\EE
has a nontrivial Lie symmetry group if and only if $h$ satisfies one of the following families of drift equations
\begin{eqnarray} \label{eqRicattidrift1theoLiesymexistiff}
b x h' - b h + \frac{1}{2} h^2 + 2 b x^{2 - \gamma} g(x) &= 2 b A x^{2-\gamma} + B \, , \\ \label{eqRicattidrift2theoLiesymexistiff}
b x h' - b h + \frac{1}{2} h^2 + 2 b x^{2 - \gamma} g(x) &= \frac{A x^{4-2 \gamma}}{2 \l( 2 - \gamma \r)^2} + \frac{B x^{2- \gamma}}{2 - \gamma} +C \, , \\ \label{eqRicattidrift3theoLiesymexistiff}
b x h' - b h + \frac{1}{2} h^2 + 2 b x^{2 - \gamma} g(x) &= \frac{A x^{4 - 2 \gamma}}{2 \l( 2 - \gamma \r)^2} + \frac{B x^{3 - \frac{3}{2} \gamma}}{3 - \frac{3}{2} \gamma} + \frac{C x^{2- \gamma}}{2 - \gamma} - \kappa \, ,
\end{eqnarray}
with $\kappa = \frac{\gamma}{8} \l( \gamma - 4 \r) b^2$ and $h(x)=x^{1- \gamma} f(x)$.
\end{theorem}
For the case $\gamma=2$, a similar result was obtained in \cite{CraddockLe09}, Proposition 2.1. Regarding the first Ricatti equation, (\ref{eqRicattidrift1theoLiesymexistiff}), the following result was described in \cite{CraddockLe09}, Theorem 3.1:

\begin{theorem} \label{theoRicattieq1symmetry} Suppose $\gamma \neq 2$ and $h(x) = x^{1- \gamma} f(x)$ is a solution of the Ricatti equation
\begin{displaymath}
b x h' - b h + \frac{1}{2} h^2 + 2 b x^{2 - \gamma} g(x) = 2 b A x^{2 - \gamma} + B \, .
\end{displaymath}
Then the PDE (\ref{eqtheoexistnontrivsym}) has a symmetry of the form
\begin{eqnarray} \label{eqtheoricatti1Ubar}
\overline{U}_{\varepsilon}(x,t) &= \frac{1}{\l( 1 + 4 \varepsilon t \r)^{\frac{1- \gamma}{2 - \gamma}}} \exp \left\{ \frac{- 4 \varepsilon \l( x^{2 - \gamma} + A b \l( 2 - \gamma \r)^2 t^2 \r) }{b \l( 2 - \gamma \r)^2 \l( 1 + 4 \varepsilon t \r)} \right\}
\\ &\quad \exp \left\{ \frac{1}{2 b} \l( F \l( \frac{x}{ \l( 1+ 4 \varepsilon t \r)^{\frac{2}{2-\gamma}}} \r) - F \l( x \r) \r) \right\}
\\ &\quad  u \l( \frac{x}{\l( 1 + 4 \varepsilon t \r)^{\frac{2}{2 - \gamma}}} , \frac{t}{1+ 4 \varepsilon t} \r) \, ,
\end{eqnarray}
where $F'(x) = f(x) / x^{\gamma}$ and $u$ is a solution of the respective PDE. That is, for $\varepsilon$ sufficiently small, $U_{\varepsilon}$ is a solution of (\ref{eqtheoexistnontrivsym}) whenever $u$ is. If $u(x,t)=u_0(x)$ with $u_0$ an analytic, stationary solution there is a fundamental solution $p(t,x,y)$ of (\ref{eqtheoexistnontrivsym}) such that
\begin{displaymath}
\int^{\infty}_0 \exp \{ - \lambda y^{2-\gamma} \} u_0 \l( y \r) p \l( t,x,y \r) dy = U_{\lambda}(x,t) \, .
\end{displaymath}
Here $U_{\lambda}(x,t) = \overline{U}_{\frac{1}{4} b \l( 2 - \gamma \r)^2 \lambda}$. Further, if $u_0=1$, then $\int^{\infty}_0 p(t,x,y)dy=1$.
\end{theorem}
For the remaining two Ricatti equations, (\ref{eqRicattidrift2theoLiesymexistiff}) and (\ref{eqRicattidrift3theoLiesymexistiff}), we refer the reader to Theorems 2.5 and 2.8 in \cite{Craddock09}.

We would now like to illustrate how the method can be used. Consider a squared Bessel process of dimension $\delta$, where $\delta \geq 2$,
\begin{displaymath}
d X_t = \delta dt + 2 \sqrt{X_t} dW_t \, ,
\end{displaymath}
where $X_0 = x >0$. The drift $f(x)=\delta$ satisfies equation (\ref{eqRicattidrift1theoLiesymexistiff}) with $A=0$. Consequently, using Theorem \ref{theoRicattieq1symmetry} with $A=0$ and $u(x,t)=1$, we obtain
\begin{displaymath}
\overline{U}_{\varepsilon}(x,t) = \exp \left\{ - \frac{4 \epsilon x}{b \l( 1 + 4 \varepsilon t\r)} \right\} \l( 1 + 4 \varepsilon t \r)^{- \frac{\delta}{b}} \, ,
\end{displaymath}
where $b=2$. Setting $\varepsilon = \frac{b \lambda}{4}$, we obtain the Laplace transform
\begin{eqnarray*}
U_{\lambda}(x,t) &=& \int^{\infty}_0 \exp \left\{ - \lambda y \right\} p(t,x,y) dy
\\ &=& \exp \left\{ - \frac{x \lambda}{1+ 2 \lambda t} \right\} \l(1 + 2 \lambda t \r)^{- \frac{\delta}{2}} \, ,
\end{eqnarray*}
which is easily inverted to yield
\begin{equation} \label{eqtransdenssqBessel}
p(t,x,y) = \frac{1}{2 t} \l( \frac{x}{y} \r)^{\frac{\nu}{2}} I_{\nu} \l( \frac{\sqrt{x y}}{t} \r) \exp \left\{ - \frac{(x+y)}{2 t} \right\} \, ,
\end{equation}
where $\nu = \frac{\delta}{2}-1$ denotes the index of the squared Bessel process. Equation (\ref{eqtransdenssqBessel}) shows the transition density of a squared Bessel process started at time $0$ in $x$ for being at time $t$ in $y$. Recall that $I_{\nu}$ denotes the modified Bessel function of the first kind. This result, together with the real world pricing formula, (\ref{eqrealworldprice}), allows us to price a wide range of European style and path-dependent derivatives with payoffs of the type $H=f(S^*_{t_1}, S^*_{t_2}, \dots, S^*_{t_d} )$, where $d \geq 1$ and $t_1 , t_2 , \dots , t_d$ are given deterministic times.

By exploiting the tractability of the underlying processes, Lie symmetry methods allow us to design efficient Monte Carlo algorithms, as the following example from \cite{BaldeauxChaPla1} and \cite{BaldeauxChaPla2} shows. We now consider the problem of pricing derivatives on realized variance. Here we define realized variance to be the quadratic variation of the log-index, and we formally compute the quadratic variation of the log-index in the form,
\begin{displaymath}
\left[ \log(S^{\deltas}_{\cdot} ) \right]_T = \int^T_0 \frac{dt}{Y_t} \, .
\end{displaymath}
Recall from Section \ref{secBMApproach} that $Y= \left\{ Y_t \, , \, t \geq 0 \right\}$ is a square-root process whose dynamics are given in equation (\ref{eqsqrootproc}). In particular, we focus on put options on volatility, where volatility is defined to be the square-root of realized variance. We remark that call options on volatility can be obtained via the put-call parity relation in Lemma 4.1 in \cite{BaldeauxChaPla1}. The real-world pricing formula (\ref{eqrealworldprice}) yields the following price for put options on volatility
\begin{equation} \label{eqputonvol}
S^{\deltas}_t E \left( \frac{(K - \sqrt{\frac{1}{T} \int^T_0 \frac{ds}{Y_s} })^+}{S^{\deltas}_T} \bigg| \cA_t \right) \, .
\end{equation}
For computing the expectation in (\ref{eqputonvol}) via Monte Carlo methods, one first needs to have access to the joint density of $(S^{\deltas}_T , \int^T_0 \frac{ds}{Y_s} )$ and subsequently perform the Monte Carlo simulation. Before presenting the relevant result, we recall that $S^{\deltas}_T = S^0_T \alpha^{\deltas}_T Y_T$, i.e. it suffices to have access to the joint distribution of $(Y_T , \int^T_0 \frac{dt}{Y_t} )$. We remark that if we have access to the Laplace transform of $(Y_T , \int^T_0 \frac{dt}{Y_t} )$, i.e.
\begin{equation} \label{eqjointPltrafo}
E \left( \exp \left( - \lambda Y_T - \mu \int^T_0 \frac{dt}{Y_t} \right) \right) \, ,
\end{equation}
then we have, in principle, solved the problem. From the point of view of implementation though, inverting a two-dimensional Laplace transform numerically is expensive. The following result from \cite{CraddockLe09}, see Corollaries 5.8 - 5.9, goes further: In fact the fundamental solution corresponds to inverting the expression in (\ref{eqjointPltrafo}) with respect to $\lambda$, which significantly reduces the computational complexity.

\begin{lemma} \label{lemLaplacejointdistrib} The joint Laplace transform of $Y_T$ and $\int^T_0 \frac{dt}{Y_t} $ is given by
\begin{eqnarray*}
\lefteqn{E \l( \exp \l( - \lambda Y_T - \mu \int^T_0 \frac{1}{Y_t} dt \r) \r)}
\\ &=& \frac{\Gamma(3/2 + \nu/2)}{\Gamma(\nu+1)} \beta x^{-1} \exp \l( \eta \l( T+ x - \frac{x}{\tanh \l( \eta T /2 \r) } \r) \r)
\\ && \frac{1}{\beta \alpha} \exp \l( \beta^2 / ( 2 \alpha)  \r) M_{-k, \nu/2} \l( \frac{\beta^2}{\alpha} \r) \, ,
\end{eqnarray*}
where $\alpha= \eta \l( 1 + \coth (\frac{\eta t}{2}) \r) +\lambda$, $\beta = \frac{\eta \sqrt{x}}{\sinh \l(\frac{\eta t}{2} \r)}$, $\nu=2\sqrt{\frac{1}{4}+2 \mu}$, and $M_{s,r}(z)$ denotes the Whittaker function of the first kind. In \cite{CraddockLe09}, the inverse with respect to $\lambda$ was already performed explicitly and is given as
\begin{eqnarray} \nonumber
p(T,x,y) &=& \frac{\eta}{\sinh \l( \eta T /2 \r)} \l( \frac{y}{x} \r)^{1/2}
\\ && \label{eqLptransformjd} \exp \l( \eta \l( T + x-y - \frac{x+y}{\tanh (\eta T/2)} \r) \r) I_{\nu} \l( \frac{2 \eta \sqrt{x y}}{\sinh \l( \eta T/2 \r)} \r) \, .
\end{eqnarray}
\end{lemma}
Consequently, to recover the joint density of $(Y_T, \int^T_0 \frac{dt}{Y_t} )$, one only needs to invert a one-dimensional Laplace transform. For further details, we refer the interested reader to \cite{BaldeauxChaPla2}. By gaining access to the relevant joint densities, this example demonstrates that Lie symmetry methods allow us to design efficient Monte Carlo algorithms for challenging finance problems.

\section{Wishart Processes} \label{secWishproc}

Very tractable and highly relevant to finance are models that generalize the previously mentioned MMM. Along these lines, in this section we discuss Wishart processes with a view towards exact simulation. As demonstrated in \cite{Bru91}, Wishart processes turn out to be the multidimensional extensions of squared Bessel processes. However, they also turn out to be affine, see \cite{GouSuf03}, and \cite{GouSuf04}. Prior to the latter two contributions, the literature was focused on affine processes taking values in the Euclidean space, see e.g. \cite{DuffiePanSin00}, and \cite{DuffieFilSch03}. Subsequently, matrix-valued affine processes were studied, see e.g. \cite{CuchieroFilMayTei11}, and \cite{GrasselliTebaldi08}. Since \cite{GouSuf03}, and \cite{GouSuf04}, it has been more widely known that Wishart processes are analytically tractable, since their characteristic function is available in closed form; see also \cite{GnoattoGrasselli11}. In this section, we exploit this fact when we discuss exact simulation of Wishart processes.

Firstly, we fix notation and present an existence result. Wishart processes are $S^+_d$ or $\overline{S^+_d}$ valued, i.e. they assume values in the set of positive definite or positive semidefinite matrices, respectively. This makes them natural candidates for the modeling of covariance matrices, as noted in \cite{GouSuf03}. Starting with \cite{GouSuf03} and \cite{GouSuf04}, there is now a substantial body of literature applying Wishart processes to problems in finance, see \cite{BurasciCiTr10}, \cite{BurasciPoTr10}, \cite{DaFonsecaGrIe08a}, \cite{DaFonsecaGrIe08b}, \cite{DaFonsecaGrTe07}, \cite{DaFonsecaGrTe08}, and \cite{GourierouxMoSu07}. In the current paper we study Wishart processes in a pure diffusion setting. For completeness, we mention that matrix valued processes incorporating jumps have been studied, see e.g. in \cite{BarndorffSt07}, and \cite{LeippoldTr08}. These processes are all contained in the affine framework introduced in \cite{CuchieroFilMayTei11}, where we direct the reader interested in affine matrix valued processes.

In the following, we introduce the Wishart process as described in the work of Grasselli and collaborators; see \cite{DaFonsecaGrTe07} and \cite{GrasselliTebaldi08}. For $\bx \in \psdmatd$, we introduce the $\psdmatd$ valued Wishart process $\bX^{\bix} = \bX = \left\{ \bX_t \, , \, t \geq 0 \right\}$, which satisfies the SDE
\begin{equation} \label{eqWishart}
d \bX_t = \left( \alpha \ba^{\top} \ba + \bb \bX_t + \bX_t \bb^{\top} \right) dt + \left( \sqrt{\bX_t} d \bW_t \ba + \ba^{\top} d \bW^{\top}_t \sqrt{\bX_t} \right) \, ,
\end{equation}
where $\alpha \geq 0$, $\bb \in \matd$, $\ba \in \matd$. Here $\matd$ denotes the set of $d \times d$ matrices taking values in $\Re$. An obvious question to ask is whether equation (\ref{eqWishart}) admits a solution, and, furthermore, if such a solution is unique and strong. For results on weak solutions we refer the reader to \cite{CuchieroFilMayTei11}, and for results on strong solutions to \cite{MayerhoferPfSt11}. We now present a summary of results, which in this form also appeared in \cite{AhdidaAlfonsi10}; see Theorem 1 in \cite{AhdidaAlfonsi10}.

\begin{theorem} \label{theoremexistWish} Assume that $\bx \in \psdmatd$, and $\alpha \geq d-1$, then equation (\ref{eqWishart}) admits a unique weak solution. If $\bx \in \pdmatd$ and $\alpha \geq d+1$, then this solution is strong.
\end{theorem}
In this paper, we are interested in exact simulation schemes to be used in Monte Carlo methods. Hence weak solutions suffice for our purposes and we assume that $\alpha>d-1$, so that the weak solution is unique. As in \cite{AhdidaAlfonsi10}, we use $WIS_d(\bx,\alpha,\bb,\ba)$ to denote a Wishart process and $\WISt$ for the value of the process at the time point $t$.

We begin with the study of some special cases, which includes an extension of the MMM to the multidimensional case. We use $\bB_t$ to denote an $n \times d$ Brownian motion and set
\begin{equation} \label{eqdefsimpleWishart}
\bX_t = \bB^{\top}_t \bB_t \, .
\end{equation}
Then it can be shown that $\bX = \left\{ \bX_t \, , \, t \geq 0 \right\}$ satisfies the SDE
\begin{displaymath}
d \bX_t = n \bI_d dt + \sqrt{\bX_t} d\bW_t + d \bW^{\top}_t \sqrt{\bX_t} \, , \end{displaymath}
where $\bW_t$ is a $d \times d$ Brownian motion, and $\bI_d$ denotes the $d \times d$ identity matrix. This corresponds to the case where we set
\begin{displaymath}
\ba = \bI_d \, , \, \bb = \vo \, , \, \alpha = n \, .
\end{displaymath}
We now provide the analogous scalar result, showing that Wishart processes generalize squared Bessel processes: Let $\delta \in \mathcal{N}$, and set
\begin{displaymath}
x= \sum^{\delta}_{k=1} (w^k)^2 \, .
\end{displaymath}
Now we set
\begin{equation} \label{eqintdimbess1}
X_t = \sum^{\delta}_{k=1} (W^k_t + w^k)^2 \, .
\end{equation}
Then $X$ can be shown to satisfy the SDE
\begin{displaymath}
d X_t = \delta dt + 2 \sqrt{X_t} d B_t \, ,
\end{displaymath}
where $B= \left\{ B_t\, , \, t \geq 0 \right\}$ is a scalar Brownian motion. This shows that (\ref{eqdefsimpleWishart}) is the generalization of (\ref{eqintdimbess1}). Furthermore, it is also clear how to simulate (\ref{eqdefsimpleWishart}).

Next, we illustrate how Wishart processes can be used to extend the MMM from Section \ref{secBMApproach}. We recall some results pertaining to matrix-valued random variables, see e.g. \cite{GuptaNag}, and \cite{Muirhead82}. We introduce some auxialiary notation. We denote by $\mathcal{M}_{m,n} (\Re)$ the set of all $m \times n$ matrices with entries in $\Re$. Next, we present a one-to-one relationship between vectors and matrices.

\begin{definition} \label{defmatvarstochmatven1to1} Let $\bA \in \mathcal{M}_{m,n}(\Re)$ with columns $\ba_i \in \Re^m$, $i=1,\dots,n$, and define the function $vec: \mathcal{M}_{m,n} (\Re) \rightarrow \Re^{m n }$ via
\begin{displaymath}
vec(\bA) = \left( \begin{array}{c} \ba_1 \\ \vdots \\ \ba_n \end{array} \right) \, .
\end{displaymath}
\end{definition}
We can now define the matrix variate normal distribution.

\begin{definition} \label{defnormdistr} A $p \times n$ random matrix is said to have a matrix variate normal distribution with mean $\bM \in \mathcal{M}_{p,n}(\Re)$ and covariance $\vSigma \otimes \vPsi$, where $\vSigma \in \mathcal{S}^+_p$, $\vPsi \in \mathcal{S}^+_n$, if $vec({\bX}^{\top}) \sim \mathcal{N}_{p n} ( vec({\bM}^{\top}), \vSigma \otimes \vPsi)$, where $\mathcal{N}_{p n}$ denotes the multivariate normal distribution on $\Re^{p n}$ with mean $vec({\bM}^{\top})$ and covariance $\vSigma \otimes \vPsi$. We will use the notation $\bX \sim \mathcal{N}_{p,n}(\bM, \vSigma \otimes \vPsi)$.
\end{definition}
Next, we introduce the Wishart distribution, which we link in the subsequent theorem to the normal distribution.

\begin{definition} \label{defWishartdistrib} A $p \times p$-random matrix $\bX$ in $\mathcal{S}^+_p$ is said to have a noncentral Wishart distribution with parameters $p \in \cN$, $n \geq p$, $\vSigma \in \mathcal{S}^+_p$ and $\vTheta \in \mathcal{M}_p(\Re)$, if its probability density function is of the form
\begin{eqnarray*} \lefteqn{f_{\bX}(\bS)}
\\ &=& \left( 2^{\frac{1}{2} n p} \Gamma_p ( \frac{n}{2} ) det (\vSigma)^{ \frac{n}{2}} \right)^{-1} etr \left( - \frac{1}{2} ( \vTheta + \vSigma^{-1} \bS )  \right)
\\ && det(\bS)^{ \frac{1}{2} ( n - p -1) } \phantom{i}_0 F_1 \left( \frac{n}{2} ; \frac{1}{4} \vTheta \vSigma^{-1} \bS \right)
\end{eqnarray*}
where $\bS \in \mathcal{S}^+_p$ and $\phantom{i}_0 F_1$ is the matrix-valued hypergeometric function, see \cite{GuptaNag}, and \cite{Muirhead82} for a definition. We write
\begin{displaymath}
\bX~\sim~\mathcal{W}_p(n, \vSigma, \vTheta) \, .
\end{displaymath}
\end{definition}
Before stating the next result, recall that scalar non-central chi-squared random variables of integer degrees of freedom, can be constructed via sums of normal random variables; see e.g. \cite{JohnsonKotBal95}. The following result presents the matrix variate analogy.

\begin{theorem} \label{theosquarenormisWish} Let $\bX \sim \mathcal{N}_{p,n}(\bM, \vSigma \otimes \bI_n)$, $n \in \left\{ p, p+1, \dots \right\}$. Then
\begin{displaymath}
\bX {\bX}^{\top} \sim \mathcal{W}_p ( n , \vSigma , \vSigma^{-1} \bM {\bM}^{\top} ) \, .
\end{displaymath}
\end{theorem}

\section{Bivariate MMM}

Theorem \ref{theosquarenormisWish} is now employed to extend the MMM to a bivariate case. We consider exchange rate options, and follow the ideas from \cite{HeathPlaten05}. The GOP denominated in units of the domestic currency is denoted by $S^a$, and the GOP denominated in the foreign currency by $S^b$. An exchange rate at time $t$ can be expressed in terms of a ratio of two GOP denominations. Then one would pay at time $t$, $\frac{S^a_t}{S^b_t}$ units of currency $a$ to obtain one unit of the foreign currency $b$. As the domestic currency is indexed by $a$, the price of, say, a call option with maturity $T$ on the exchange rate can be expressed via the real world pricing formula (\ref{eqrealworldprice}) as:
\begin{equation} \label{eqpriceFXcall}
S^a_0 E \l( \frac{\l( \frac{S^a_T}{S^b_T} - K \r)^+}{S^a_T} \r) \, .
\end{equation}
We now discuss a bivariate extension of the MMM from Section \ref{secBMApproach}, which is still tractable, as we can employ the non-central Wishart distribution to compute (\ref{eqpriceFXcall}). For $k \in \left\{ a , b \right\}$, we set
\begin{displaymath}
S^k_t = S^{0,k}_t \barS^k_t \, ,
\end{displaymath}
where $S^{0,k}_t = \exp \{ r_k t \}$, $S^{0,k}_0 = 1$, so $S^{0,k}$ denotes the savings account in currency $k$, which for simplicity is assumed to be a deterministic exponential function of time. As for the stylized MMM, we model the discounted GOP, $\bar{S}^k_t$, denominated in units of the $k$th savings account, $S^{0,k}_t$, as a time-changed squared Bessel process of dimension four. We introduce the $2 \times 4$ matrix process $\bX=\left\{ \bX_t \, , \, t \geq 0 \right\}$ via
\begin{displaymath}
\bX_t = \left[ \begin{array}{cccc} \l( W^{1,1}_{\varphi^1(t)} + w^{1,1} \r) & \l( W^{2,1}_{\varphi^1(t)} + w^{2,1} \r) & \l( W^{3,1}_{\varphi^1(t)} + w^{3,1} \r) & \l( W^{4,1}_{\varphi^1(t)} + w^{4,1} \r) \\ \l( W^{1,2}_{\varphi^2(t)} + w^{1,2} \r) & \l( W^{2,2}_{\varphi^2(t)} + w^{2,2} \r) & \l( W^{3,2}_{\varphi^2(t)} + w^{3,2} \r) & \l( W^{4,2}_{\varphi^2(t)} + w^{4,2} \r) \end{array} \right] \, .
\end{displaymath}
The processes $W^{i,1}_{\varphi^1}$, $i=1,\dots,4$, denote independent Brownian motions, subject to the deterministic time-change
\begin{displaymath}
\varphi^1(t) = \frac{\alpha^1_0}{4 \eta^1} \l( \exp \{ \eta^1 t \} - 1 \r) = \frac{1}{4} \int^t_0 \alpha^1_s ds \, ,
\end{displaymath}
c.f. Section \ref{secBMApproach}. Similarly, also $W^{i,2}_{\varphi^2}$, $i=1,\dots,4$, denote independent Brownian motions, subject to the deterministic time change
\begin{displaymath}
\varphi^2(t) = \frac{\alpha^2_0}{4 \eta^2} \l( \exp \{ \eta^2 t \} -1 \r) = \frac{1}{4} \int^t_0 \alpha^2_s ds \, .
\end{displaymath}
Now, consider the process $\bY=\left\{ \bY_t \, , \,  t \geq 0 \right\}$, which assumes values in $S^+_2$, and is given by
\begin{displaymath}
\bY_t := \bX_t {\bX}^{\top}_t \, , \, t \geq 0 \, ,
\end{displaymath}
which yields
\begin{eqnarray*}
\lefteqn{\bY_t =}
\\& \left[ \begin{array}{cc} \sum^4_{i=1} \l( W^{i,1}_{\varphi^1(t)} + w^{i,1} \r)^2 & \sum^4_{i=1} \sum^2_{j=1} \l( W^{i,j}_{\varphi^j(t)} + w^{i,j} \r)  \\ \sum^4_{i=1} \sum^2_{j=1} \l( W^{i,j}_{\varphi^j(t)} + w^{i,j} \r)  & \sum^4_{i=1} \l( W^{i,2}_{\varphi^2(t)} + w^{i,2} \r)^2 \end{array} \right] \, .
\end{eqnarray*}
We set
\begin{displaymath}
\barS^{a}_t = Y^{1,1}_t \, ,
\end{displaymath}
and
\begin{displaymath}
\barS^b_t = Y^{2,2}_t \, ,
\end{displaymath}
so we use the diagonal elements of $\bY_t$ to model the GOP in different currency denominations. Next, we introduce the following dependence structure: The Brownian motions $W^{i,1}$ and $W^{i,2}$, $i=1,\dots,4$, covary as follows,
\begin{equation}\label{eqdefdependencetwodimMMM}
\langle W^{i,1}_{\varphi^1 (\cdot)} , W^{i,2}_{\varphi^2 (\cdot)}  \rangle_t = \frac{\varrho}{4} \int^t_0 \sqrt{ \alpha^1_s \alpha^2_0 } ds, i=1,\dots,4 \, ,
\end{equation}
where $-1 < \varrho <1$. The specification (\ref{eqdefdependencetwodimMMM}) allows us to employ the non-central Wishart distribution; we work through this example in detail, as it illustrates how to extend the stylized MMM to allow for a non-trivial dependence structure, but still exploit the tractability of the Wishart distribution. We recall that $vec({\bX}^{\top}_T)$ stacks the two columns of ${\bX}^{\top}_T$, hence
\begin{displaymath}
vec( {\bX}^{\top}_T ) = \left[ \begin{array}{c} \l( W^{1,1}_{\varphi^1(T)} + w^{1,1} \r) \\ \vdots \\ \l( W^{4,1}_{\varphi^1(T)} + w^{4,1} \r) \\ \l( W^{1,2}_{\varphi^2(T)} + w^{1,2} \r) \\ \vdots \\ \l( W^{4,2}_{\varphi^2(T)} + w^{4,2} \r) \end{array}  \right] \, .
\end{displaymath}
It is easily seen that the mean matrix $\bM$ of $vec(X^{\top}_T)$ satisfies
\begin{equation} \label{eqmeanmatrix}
vec \l( {\bM}^{\top} \r) = \left[ \begin{array}{c} w^{1,1} \\ \vdots \\ w^{4,1} \\ w^{1,2} \\ \vdots \\ w^{4,2}  \end{array} \right]
\end{equation}
and the covariance matrix of $vec({\bX}^{\top}_T)$ is given by
\begin{equation} \label{eqcovarmatrix}
\vSigma \otimes \bI_4 = \left[ \begin{array}{cc} \Sigma^{1,1} \bI_4 & \Sigma^{1,2} \bI_4 \\ \Sigma^{2,1} \bI_4 & \Sigma^{2,2} \bI_4 \end{array} \right] \, ,
\end{equation}
where $\vSigma$ is a $2 \times 2$ matrix with $\Sigma^{1,1} = \varphi^1(T)$, $\Sigma^{2,2} = \varphi^2(T)$, and
\begin{displaymath}
\Sigma^{1,2} = \Sigma^{2,1} = \frac{\varrho}{4}\int^t_0 \sqrt{ \alpha^1_s \alpha^2_s }ds \, .
\end{displaymath}
We remark that assuming $-1 < \varrho < 1$ results in $\vSigma$ being positive definite. It now immediately follows from Theorem \ref{theosquarenormisWish} that
\begin{displaymath}
\bX_T {\bX}^{\top}_T \sim W_2 \l( 4 , \vSigma , \vSigma^{-1} \bM \bM ^{\top} \r) \, ,
\end{displaymath}
where $\bM$ and $\vSigma$ are given in equations (\ref{eqmeanmatrix}) and (\ref{eqcovarmatrix}), respectively. Recall that we set
\begin{eqnarray*}
\bY_t &=& \bX_t {\bX}^{\top}_t \, ,
\\ \barS^a_t  &=& Y^{1,1}_t \, ,
\\ \barS^b_t &=& Y^{2,2}_t \, ,
\end{eqnarray*}
hence we can compute (\ref{eqpriceFXcall}) using
\begin{displaymath}
E \l( f (\bY_T) \r) \, ,
\end{displaymath}
where $f : S^+_2 \rightarrow \Re$ is given by
\begin{displaymath}
f (y) = \frac{\l( \frac{\exp \{ r_1 T \} y^{1,1} }{ \exp \{ r_2 T \} y^{2,2} } - K \r)^+}{\exp \{ r_1 T \} y^{1,1}} \, ,
\end{displaymath}
for $y \in S^+_2$, and $y^{i,i}$, $i=1,2$, are the diagonal elements of $y$, and the probability density function of $\bY_T$ is given in Definition \ref{defWishartdistrib}.

We now discuss further exact simulation schemes for Wishart processes, where we rely on \cite{AhdidaAlfonsi10} and \cite{BenabidBeEl10}. For integer valued parameters $\alpha$ in (\ref{eqWishart}), we have the following exact simulation scheme, which generalizes a well-known result from the scalar case, linking Ornstein-Uhlenbeck and square-root processes. In particular, this lemma shows that, in principle, certain square-root processes can be simulated using Ornstein-Uhlenbeck processes.
\begin{lemma} Let $A>0$, $Q>0$, and define the SDEs
\begin{displaymath}
dX^i_t = - A X^i_t dt + Q dW^i_t \, ,
\end{displaymath}
for $i= 1, \dots , \beta$, where $\beta \in \mathcal{N}$, $W^1, W^2, \dots, W^{\beta}$ are independent Brownian motions. Then
\begin{displaymath}
Z_t = \sum^{\beta}_{i=1} (X^i_t)^2
\end{displaymath}
is a square-root process of dimension $\beta$, whose dynamics are characterised by an SDE
\begin{displaymath}
d Z_t = ( \beta Q^2 - 2 A Z_t ) dt + 2 Q \sqrt{Z_t} dB_t \, ,
\end{displaymath}
where $B$ is a resulting Brownian motion.
\end{lemma}
\begin{proof} The proof follows immediately from the It\^o-formula. 
\qed
\end{proof}
This result is easily extended to the Wishart case, for integer valued $\alpha$, see Section 1.2.2 in \cite{BenabidBeEl10}. We define
\begin{equation} \label{eqWishartfromOU}
\bV_t = \sum^{\beta}_{k=1} \bX_{k,t} \bX^{\top}_{k,t} \, ,
\end{equation}
where
\begin{equation} \label{eqvecOUforWIshart}
d \bX_{k,t} = A \bX_{k,t} dt + \bQ^{\top} d \bW_{k,t} \, , k=1, \dots , \beta \, ,
\end{equation}
where $A \in \matd$, $\bX_t \in \Re^d$, $\bQ \in \matd$, $\bW_k \in \Re^d$, so that $\bV_t \in \matd$. The following lemma gives the dynamics of $\bV= \left\{ \bV_t \, , \, t \geq 0 \right\}$.

\begin{lemma} \label{lemdynamicsWishartfromOU} Assume that $\bV_t$ is given by equation (\ref{eqWishartfromOU}), where $\bX_t$ satisfies equation (\ref{eqvecOUforWIshart}). Then
\begin{displaymath}
d \bV_t = \left( \beta \bQ^{\top} \bQ + A \bV_t + \bV_t A^{\top} \right) dt + \sqrt{\bV_t} d \bW_t \bQ + \bQ^{\top} d\bW^{\top}_t \sqrt{\bV_t} \, ,
\end{displaymath}
where $\bW = \left\{ \bW_t \, , \, t \geq 0 \right\}$ is a $d \times d$ matrix valued Brownian motion that is determined by
\begin{displaymath}
\sqrt{\bV_t} d \bW_t = \sum^{\beta}_{k=1} \bX_{k,t} d \bW^{\top}_{t,k} \, .
\end{displaymath}
\end{lemma}
Finally, we remind the reader that vector-valued Ornstein-Uhlenbeck processes can be simulated exactly, see e.g. Chapter 2 in \cite{PB10}.

For the general case, we refer the reader to \cite{AhdidaAlfonsi10}. In that paper, a remarkable splitting property of the infinitesimal generator of the Wishart process was employed to come up with an exact simulation scheme for Wishart processes without any restriction on the parameters. Furthermore, in \cite{AhdidaAlfonsi10} higher-order discretization schemes for Wishart processes and second-order schemes for general affine diffusions on positive semidefinite matrices were presented. These results emphasize that Wishart processes are suitable candidates for financial models, since exact simulation schemes are readily available.

\section{Conclusion} \label{secconc}

In this paper, we discussed classes of stochastic processes for which exact simulation schemes are available. In the one-dimensional case, our first theorem gives access to explicit transition densities via Lie symmetry group results. In the multidimensional case the probability law of Wishart processes is described explicitly. When considering applications in finance, one needs a framework that can accommodate these processes as asset prices, in particular, when they generate strict local martingales. We demonstrated that the benchmark approach is a suitable framework for these processes and allows to systematically exploit the tractability of the models described. For long dated contracts in finance, insurance and for pensions the accuracy of the proposed simulation methods is extremely important.

\end{document}